\documentclass[11.05pt,a4paper]{amsart}
\usepackage{amssymb,amsmath,amsfonts,amscd,enumerate,verbatim,calc,amsthm}
\usepackage{latexsym}
\usepackage{amsthm,amsfonts,amssymb,mathrsfs}
\usepackage{rotating}
\usepackage[leqno]{amsmath}
\usepackage{xspace}
\usepackage[all]{xy}
\usepackage{longtable}
\usepackage[colorlinks=true]{hyperref}
\usepackage[all]{xy}
\input xy
\xyoption{all}

\headsep=1cm \numberwithin{equation}{section}
\newtheorem{thm}{Theorem}[section]
\newtheorem{cor}[thm]{Corollary}
\newtheorem{lem}[thm]{Lemma}
\newtheorem{prop}[thm]{Proposition}
\newtheorem{defn}[thm]{Definition}

\newtheorem{defrem}[thm]{Definition and Remarks }
\newtheorem{Con*}[thm]{Conjectuer}

\newcommand{\Ann}{\mbox{Ann}\,}

\newcommand{\Ker}{\mbox{Ker}\,}

\newcommand{\Tot}{\mbox{Tot}\,}
\newcommand{\Nd}{\mbox{Ndim}\,}

\newcommand{\h}{\mbox{ht}\,}

\newcommand{\E}{\mbox{E}}

\newcommand{\uhom}{{\mathbf R}\Hom}
\newcommand{\utp}{\otimes^{\mathbf L}}
\newcommand{\ugamma}{{\mathbf R}\Gamma}
\newcommand{\Lam}{{\mathbf L}\Lambda}
\renewcommand{\H}{\mbox{H}}
\newcommand{\V}{\mbox{V}}

\newcommand{\D}{\mbox{D}}

\newcommand{\fa}{\mathfrak{a}}
\newcommand{\fb}{\mathfrak{b}}

\newcommand{\fm}{\mathfrak{m}}

\newcommand{\fn}{\mathfrak{n}}

\renewcommand{\Im}{\mbox{Im}\,}

\def\Id{\operatorname{\mathsf{Id}}}
\def\id{\operatorname{\mathsf{id}}}

\def\wid{\operatorname{\mathsf{width}}}

\def\fd{\operatorname{\mathsf{fd}}}

\def\cd{\operatorname{\mathsf{cd}}}

\def\T{\operatorname{\mathsf{T}}}

\def\Ext{\operatorname{\mathsf{Ext}}}
\def\gr{\operatorname{\mathsf{grade}}}
\def\depth{\operatorname{\mathsf{depth}}}
\def\Hom{\operatorname{\mathsf{Hom}}}
\def\dim{\operatorname{\mathsf{dim}}}

\def\Tor{\operatorname{\mathsf{Tor}}}

\DeclareMathOperator{\Supp}{Supp}

\begin{document}

\title[]
{some characterizations of Relative Sequentially Cohen-Macaulay and Relative Cohen-Macaulay modules }

\author[M. Rahro Zargar]{Majid Rahro Zargar}


\address{Majid Rahro Zargar, Department of Engineering Sciences, Faculty of Advanced Tech-
nologies, University of Mohaghegh Ardabili, Namin, Ardabil, Iran,}
\email{zargar9077@gmail.com}
\email{m.zargar@uma.ac.ir}

\begin{abstract}Let $M$ be an $R$-module over a Noetherian ring $R$ and $\fa$ be an ideal of $R$ with $c=\cd(\fa,M)$. First, we prove that
 $M$ is finite $\fa$-relative Cohen-Macaulay if and only if $\H_{i}(\Lam_{\fa}(\H_{\fa}^c(M)))=0$ for all $i\neq c$ and $\H_{c}(\Lam_{\fa}(\H_{\fa}^c(M)))\cong
 \widehat{M}^{\fa}$. Next, over an $\fa$-relative Cohen-Macaulay local ring $(R,\fm)$, we provide a characterization of $\fa$-relative sequentially Cohen-Macaulay modules $M$  in terms of $\fa$-relative Cohen-Macaulayness of the $R$-modules $\Ext^{d-i}_{R}(M,\D_{\fa})$ for all $i\geq 0$, where $\D_{\fa}=\Hom_R(\H^d_{\fa}(R),\E(R/\fm))$ and $d=\cd(\fa,R)$. Finally, we provide another characterization of $\fa$-relative sequentially Cohen-Macaulay modules $M$ in terms of vanishing of the local homology modules $\H_{j}(\Lam_{\fa}(\H_{\fa}^i(M)))=0$ for all $0\leq i\leq c$ and for all $j\neq i$.

\end{abstract}
\maketitle


\section{Introduction}
Throughout this paper, $R$ is a commutative Noetherian ring and $\fa$ is a proper ideal of $R$. In the case where $R$ is local
 with the maximal ideal $\fm$, $\widehat{R}$ denotes the $\fm$-adic completion of $R$, $\E_R(R/\fm)$ denotes the injective hull of the residue field $R/\fm$ and $\widehat{R}^{\fa}$ denotes the $\fa$-adic completion of $R$. Also, for an $R$-module $M$, the $R$-modules $\H^i_{\fa}(M)$ for all $i$ denote the local cohomology modules, and $\H_{i}(\Lam_{\fa}(M))$ for all $i$ denote the left derived functor of the $\fa$-adic completion $\Lambda_{\fa}(-)$. The concept of sequentially Cohen- Macaulay was defined by combinatorial commutative algebraists \cite[3.9]{STA} to answer a basic question to find a non-pure
generalization of the concept of a Cohen-Macaulay module so that the face ring of a shellable (non-pure) simplicial
complex has this property. This concept was then applied by commutative algebraists to study some algebraic invariants or
special algebras that come from graphs. Recall that a finitely generated $R$-module $M$ over a local ring $R$ is called sequentially Cohen-Macaulay if there exists a finite filtration

$0=M_0 \subset M_1 \subset\dots \subset M_{r-1}\subset M_{r}=M$ of $M$ by submodules of $M$ such that each quotient module $M_{i}/M_{i-1}$ is a Cohen-Macaulay $R$-module and that $\dim(M_1 /M_0)<\dim(M_2 /M_1)<\dots<\dim(M_{r-1} /M_{r-2})<\dim(M/M_{r-1})$( For more details see \cite{SC3}).
There are some nice characterizations of the concept of sequentially Cohen-Macaulay modules which one of them is the following basic theorem of J. Herzog and D. Popescu in \cite[Theorem 2.4]{HEP}:

\begin{thm}Let $(R,\fm)$ be a $d$-dimensional Cohen-Macaulay local ring with the canonical $R$-module $\omega_{R}$ and $M$ be a finitely generated $R$-module. Then, the following statements are equivalent:
\begin{itemize}
\item[(i)]{{$M$ is sequentially Cohen-Macaulay.}}
\item[(ii)]{{$\Ext^{d-i}_{R}(M,\omega_{R})$ is zero or Cohen-Macaulay of dimension $i$ for all $i\geq 0.$}}
\end{itemize}
\end{thm}
The author, in \cite{MZ}, introduced the notion of relative Cohen-Macaulay modules with respect to an ideal $\fa$ as a natural generalization of the concept of Cohen-Macaulay modules over a local ring. Indeed, we say that an $R$-module $M$ is $\fa$-relative Cohen-Macaulay if there is precisely one non-vanishing local cohomology module of $M$ with respect to $\fa$. Also, in \cite{MZ2}, we studied the $R$-module $\D_{\fa}=\Hom_R(\H^c_{\fa}(R),\E(R/\fm))$ in the case where $R$ is a relative Cohen-Macaulay local ring with respect to $\fa$ and $c=\h_{R}\fa$. Indeed, we showed that these modules behave like the canonical modules over Cohen-Macaulay local rings. On the other hand, as a generalization of the notion of sequentially Cohen-Macaulay, the concept of relative sequentially Cohen-Macaulay was defined and studied by some authors ( see \cite{ASN}, \cite{ARH}, \cite{ARH1} and for definition see \ref{*33}). Therefore, in section 3, among other things we were able to provide the following result which can be considered a suitable generalization of the above-mentioned theorem of J. Herzog and D. Popescu.
\begin{thm}Let $(R,\fm)$ be an $\fa$-relative Cohen-Macaulay local ring with $\cd(\fa,R)=c$. Then, the following statements hold:
 \begin{itemize}
\item[(i)]{{ If $M$ is finite $\fa$-relative sequentially Cohen-Macaulay, then $\Ext^{c-i}_{R}(M,\D_{\fa})$ is zero or is $\fa$-relative Cohen-Macaulay of cohomological dimension $i$ for all $i\geq 0$.}}
\item[(ii)]{{ If for all $i\geq 0$, $\Ext^{c-i}_{R}(M,\D_{\fa})$ is zero or is $\fa$-relative Cohen-Macaulay of cohomological dimension $i$ and \emph{$\H_{\fa}^j(\Ext^{c-i}_{R}(\Ext^{c-i}_{R}(M,\D_{\fa}),\D_{\fa}))=0$} for all $j>i$, then $M\otimes_R\widehat R$ is $\fa$-relative sequentially Cohen-Macaulay.}}
\end{itemize}
\end{thm}
R. N. Roberts, in \cite{ROB}, as a dual of the notion of Krull-dimension, introduced the Noetherian dimension, $\Nd M$, for an Artinian $R$-module $M$.
On the other hand, the concept of co-regular sequence for Artinian $R$-module $M$ was introduced by E. Matlis in \cite{EM} as a natural dual of the concept of regular sequence. The width of $M$, denoted by $\wid(M)$, is the length of any maximal $M$-co-regular sequence in $\fm$. In general, for any Artinian $R$-module M $\wid M\leq \Nd M <\infty$. Z. Tang and H. Zakeri, in \cite{TZA}, introduced the concept of co-Cohen-Macaulay for Artinian $R$-module $M$. Indeed, they defined that $M$ is \textit{co-Cohen-Macaulay} if and only if $\wid M=\Nd M$. Also, they proved that $\wid M=\inf\{~i | \Tor^R_{i}(R/\fm,M)\neq 0 \}$. Furthermore, in view of \cite[Corollary 3.4.2]{SC2} one has $\wid M=\inf\{~i |~\H^{\fm}_{i}(M)\neq 0 \}$. Also, in \cite[Propositions 4.8, 4.10]{CN}, it was proved that $\Nd M = \sup\{~i |~\H^{\fm}_i(M)\neq 0 \}$ whenever $M$ is an Artinian module over a local ring $(R,\fm)$. Therefore, an Artinian $R$-module $M$ is co-Cohen-Macaulay if and only if $\H^{\fm}_i(M)=0$ for all $i\neq \Nd M$. Hence, in section 4, as a generalization of this concept our main aim is to study the $R$-modules $X$ such that its local homology modules with respect to an arbitrary ideal $\fa$, $\H_{i}(\Lam_{\fa}(X))$, has just one non-vanishing point. Indeed, we are going to provide some characterizations for relative Cohen-Macaulay and relative sequentially Cohen-Macaulay modules in terms of local homology modules. In this direction, first we provide the two lemmas \ref{00} and \ref{000}, to prove the first main theorem \ref{2} as follows:

Let $\fa$ be a proper ideal of $R$ and $M$ be an $R$-module with $\cd(\fa,M)=c$. Then, we have
 $M$ is finite $\fa$-relative Cohen-Macaulay if and only if \[\H_{i}(\Lam_{\fa}(\H_{\fa}^c(M))\cong\begin{cases}
       \widehat{M}^{\fa} & \text{if $i=c$}\\
       0& \text{otherwise},
       \end{cases} \]
Z. Tang, in \cite{ZT}, proved that over a local ring $(R,\fm)$ if $M$ is a $d$-dimensional Cohen-Macaulay $R$-module, then the $R$-module $\H_{\fm}^d(M)$ is co-Cohen-Macaulay, but the converse is not true, in general. Therefore, the above our result is a generalization of this result of Z. Tang. Also, it shows that the converse of \cite[Theorem 1.2]{SC1} is also true. Next, in Theorem \ref{fff} which is another main result in this section we provide a characterization of a relative sequentially Cohen-Macaulay as follows:
\begin{thm} Let $\fa$ be an ideal of $R$ and $M$ be a finitely generated $R$-module and set \emph{$\T^{i}(-):=\H^{\fa}_i(-)$} or \emph{$\T^{i}(-):=\H_{i}(\Lam_{\fa}(-))$}. Then, consider the following statements:
\begin{itemize}
\item[(i)]{{If $M$ is a finite $\fa$-relative sequentially Cohen-Macaulay, then \emph{$\T^j(\H_{\fa}^i(M))=0$} for all $0\leq i\leq \cd(\fa,M)$ and for all $j\neq i$ and \emph{$\T^i(\H_{\fa}^i(M))$} is zero or $\fa$-relative Cohen-Macaulay of cohomological dimension $i$.}}
\item[(ii)]{{ If \emph{$\H_{j}(\Lam_{\fa}(\H_{\fa}^i(M)))=0$} for all $0\leq i\leq \cd(\fa,M)$ and for all $j\neq i$, then $\widehat{M}^{\fa}$ is a finite $\fa$-relative sequentially Cohen-Macaulay as an $\widehat{R}^{\fa}$-module.}}
\end{itemize}
\end{thm}
In \cite[Theorem 10.5.9]{SC2}, it is shown that the local ring $(R,\fm)$ is a $d$-dimensional Gorenstein ring if and only if $\H_{i}(\Lam_{\fm}(\E_R(k)))=0$ for all $i\neq d$ and $\H_{d}(\Lam_{\fm}(\E_R(k)))\cong \widehat{R}^{\fm}.$ So, in this direction, we provide the following result which is an improvement of the above result.
\begin{thm}Let $(R,\fm)$ be a local ring of dimension $d$ and let $\fa$ be an ideal of $R$. Then, the following statements are equivalent:
\begin{itemize}
\item[(i)]{{\emph{$R$} is Gorenstein}}
\item[(ii)]{{\emph{$\H_{i}(\Lam_{\fm}(\E_R(k)))=0$} for all $i\neq d$ and \emph{$\H_{d}(\Lam_{\fm}(\E_R(k)))\cong \widehat{R}^{\fm}.$}}}
\item[(iii)]{{There exists an ideal $\fa$ of $R$ with $\cd(\fa, R)=c$ such that \emph{$\H_{i}(\Lam_{\fa}(\E_R(k)))=0$} for all $i\neq c$ and \emph{$\fd_R(\H_{c}(\Lam_{\fa}(\E_R(k))))<\infty.$}}}
\end{itemize}
\end{thm}
\section{Prerequisites}
An $R$-complex $X$ is a sequence of $R$-modules $(X_v)_{v\in\mathbb{Z}}$ together
with $R$-linear maps $(\partial_v^{X}:X_v \longrightarrow X_{v-1})_{v\in\mathbb{Z}}$,
$$X =\cdots\longrightarrow X_{v+1}\stackrel{\partial_{v+1}^X} \longrightarrow X_{v}\stackrel{\partial_{v}^X}\longrightarrow X_{v-1}\longrightarrow\cdots,$$
such that $\partial_{v}^{X}\partial_{v+1}^{X}=0$ for all $v\in\mathbb{Z}$. The derived category of $R$-modules are denoted by $\mathrm{D}(R)$ and we use the symbol $\simeq$ for denoting isomorphisms in $\mathrm{D}(R)$. Any $R$-module
$M$ can be considered as a complex having $M$ in its $0$-th spot and $0$ in its other spots. We denote the
full subcategory of homologically left (resp. right) bounded complexes by $\mathrm{D}_{\sqsubset}(R)$ (resp.
$\mathrm{D}_{\sqsupset}(R)$). Let ${X}\in \mathrm{D}(R)$ and/or ${Y}\in \mathrm{D}(R)$.
The left-derived tensor product complex of ${X}$ and ${Y}$ in $\mathrm{D}(R)$ is
denoted by ${X}\otimes_R^{{\bf L}}{Y}$ and is defined by
{$${X}\otimes_R^{{\bf L}}{Y}\simeq {F}\otimes_R{Y}\simeq{X}
\otimes_R{F}^{'}\simeq {F}\otimes_R{F},^{'}$$} where ${F}$ and
${F}^{'}$ are flat resolutions of ${X}$ and ${Y}$, respectively.
Also, let ${X}\in \mathrm{D}(R)$ and/or ${Y}\in \mathrm{D}(R)$. The right derived homomorphism complex of ${X}$ and ${Y}$ in
$\mathrm{D}(R)$ is denoted by ${\bf R}\Hom_R({X},{Y})$ and is defined by $${\bf R}\Hom_R({X},{Y})\simeq \Hom_R({P},{Y})\simeq
\Hom_R({X},{I})\simeq \Hom_R({P},I),$$ where ${P}$ and
${I}$ are the projective resolution of ${X}$ and injective resolution of ${Y}$,
respectively. For any two complexes $X$ and $Y$, we set $\Ext_R^i(X,Y)=\H_{-i}(\uhom_R(X,Y))$ and $\Tor_{i}^R(X,Y)=\H_{i}(X\utp_{R}Y)$. For any integer $n$, the $n$-fold shift of a complex $(X,\xi^X)$ is the complex $\Sigma^nX$ given by $(\Sigma^n X)_v=X_{v-n}$ and $\xi_{v}^{\Sigma^nX}=(-1)^n\xi_{v-n}^{X}$. Also, we have $\H_i(\Sigma^nX)=\H_{i-n}(X)$. Next, for any contravariant, additive and exact functor $\T: \mathrm{D}(R)\longrightarrow \mathrm{D}(R)$ and $X\in \mathrm{D}(R)$, we have $\H_l(\T(X))\cong\T(\H_{-l}(X))$ for all $l\in\mathbb{Z}$. For any
$R$-module $M$ and ideal $\fa$ of $R$, set $\Gamma_{\fa}(M):=\{x\in M|~\Supp_RRx\subseteq \V(\fa)\}.$ Now, for any $R$-complex $X$ in $\mathrm{D}_{\sqsubset}(R)$, the right derived functor of the functor $\Gamma_{\fa}(-)$ in $\mathrm{D}(R)$, $\ugamma_{
\fa}(X)$, exists and is defined by ${\bf R} \Gamma_{\fa}(X):=\Gamma_{\fa}(I)$, where $I$
is any injective resolution of $X$.  Also, for any integer $i$, the $i$-th local cohomology module of $X$ with
respect to $\fa$ is defined by $\H_{\fa}^i(X):=\H_{-i}({\bf R}\Gamma_{\fa}(X))$.

Let $\fa$ be an ideal of $R$ and $M$ an $R$-module. The i-th local homology module $\H_i^{\fa}(M)$ of $M$ with respect to $\fa$ is defined by
$$\H_i^{\fa}(M)=\underset{n\in\mathbb{N}}\varprojlim\Tor^R_{i}(R/\fa^n, M)=\underset{n\in\mathbb{N}}\varprojlim\H_{i}(R/\fa^n\otimes_R F_{\bullet}),$$
where $F_{\bullet}$ is a projective resolution of $M$. On the other hand, we use $\Lambda_{\fa}(M):=\underset{n\in\mathbb{N}}\varprojlim(R/\fa^n\otimes_{R}M)$
to denote the $\fa$-adic completion of $M$. It is known that the functor of the $\fa$-adic completion, $\Lambda_{\fa}(-)$, is an additive covariant
functor from the category of $R$-modules and $R$-homomorphisms to itself. We denote by $\H_{i}(\Lam_{\fa}(M))$ the i-th left derived module of $\Lambda_{\fa}(M)$. Since the tensor functor is not left exact and the inverse limit is not right exact on the category of $R$-modules, the functor $\Lambda_{\fa}(-)$ is neither left nor right exact. Therefore $\H_{0}(\Lam_{\fa}(M))\neq \Lambda_{\fa}(M)$,  in general. Notice that $\H_{i}(\Lam_{\fa}(M))=\H_i(\underset{n\in\mathbb{N}}\varprojlim(R/\fa^n\otimes_R F_{\bullet}))$. Thus, there are natural maps $\phi_{i}: \H_{i}(\Lam_{\fa}(M))\longrightarrow\H_i^{\fa}(M),$ which are epimorphisms by \cite[(1.1)]{GMMM}. Also, in view of \cite[Proposition 4.1]{CN}, these maps are isomorphisms provided $M$ is Artinian. Clearly, $\H_0^{\fa}(M)\cong\Lambda_{\fa}(M)$. Moreover, if $M$ is a finitely generated module, then by \cite[Remark 2·1(iii)]{CN} for all $i > 0$,  $\H_{i}(\Lam_{\fa}(M))=0$ and so $\H_i^{\fa}(M)=0$. Let $\underline{x} = x_1,\dots ,x_t$ denote a system of elements in $R$ and $\fa =(\underline{x})R$. Then, we recall the definition of the $\check{C}$ech, $\check{C}_{\underline{x}}$, and its free resolution $\check{L}_{\underline{x}}$ as done in \cite[ 6.2.2 and 6.2.3]{SC2}. For any $R$-module $M$, one has $\H_i(\Lam_{\fa}(M))=\H_i(\Hom_R(\check{L}_{\underline{x}},M))$ (see \cite{SC2} for more details and generalizations).

\begin{defn} \emph{We say that an $R$-module $M$ is \textit{relative Cohen-Macaulay} with respect to $\fa$ if there is precisely one non-vanishing local cohomology module of $M$ with respect to $\fa$. In the case where $M$ is finitely generated, clearly this is the case if and only if $\gr(\fa,M)=\cd(\fa,M)$, where $\cd(\fa,M)$ is the largest integer $i$ for which $\H_{\fa}^i(M)\neq0$. For convenience, we use the notation $\fa$-relative Cohen-Macaulay, for an $R$-module which is relative Cohen-Macaulay with respect to $\fa$. Also, in the case where $(R, \fm)$ is a relative Cohen-Macaulay local ring with respect to $\fa$ and $c=\h_{R}\fa$, we define the $R$-module $\D_{\fa}:=\Hom_R(\H_{\fa}^c(R),\E(R/\fm))$ as a relative canonical $R$-module.}
\end{defn}
Observe that the above definition provides a generalization of the concept of Cohen-Macaulay modules. Also, notice that the notion of relative Cohen-Macaulay modules is connected with the notion of cohomologically complete intersection ideals which has been studied in \cite{HSC1} and has led to some interesting results. Furthermore, recently, such modules have been studied in \cite{HSC}, \cite{MZ} and \cite{MZ1}. In \cite{MZ2}, we showed the $R$-module  $\D_{\fa}$ treat like canonical modules over Cohen-Macaulay local rings. Indeed, we provided the following result which will be used in section 3 of the present paper.

\begin{thm}\label{3}Let $(R,\fm)$ be an $\fa$-relative Cohen-Macaulay local ring with \emph{$\h_{R}\fa=c$}. Then, the following statements hold:
  \begin{itemize}
\item[(i)]{\emph{For all ideals $\fb$ of $R$ such that $\fa\subseteq\fb$, $\H_{\fb}^i(\D_{\fa})=0$ if and only if $i\neq c$. }}
\item[(ii)]{\emph{$\id_{R}(\D_{\fa})=c.$}}
\item[(iii)]{\emph{$\widehat{R}\cong\Hom_{R}(\D_{\fa},\D_{\fa})$ and $\Ext_{R}^i(\D_{\fa},\D_{\fa})=0$ for all $i>0$.}}
\item[(iv)]{\emph{For all $t= 0, 1,\ldots,c $ and for all finitely generated $\fa$-relative Cohen-Macaulay $R$-modules $M$ with $\cd(\fa,M)=t$, one has:
\begin{itemize}
\item[(a)]{{$\Ext_{R}^{c-i}(M,\D_{\fa})=0$ if and only if $i\neq t$.}}
\item[(b)]{{$\H_{\fa}^i(\Ext_{R}^{c-t}(M,\D_{\fa}))=0$ if and only if $i\neq t$}}
\end{itemize}}}
\end{itemize}
\end{thm}
\begin{defn}\label{*33}\emph{Let $M$ be an $R$-module and $\fa$ be an ideal of $R$. Then we say that $M$ is \textit{sequentially Cohen-Macaulay }with respect to $\fa$ (Abb. $\fa$-sequentially Cohen-Macaulay) if there exists a finite filtration
$$0=M_0 \subset M_1 \subset\dots \subset M_{r-1}\subset M_{r}=M$$
of $M$ by submodules of $M$ such that each quotient module $M_{i}/M_{i-1}$ is an $\fa$-relative Cohen-Macaulay $R$-module and that $\cd(\fa,M_1 /M_0)<\cd(\fa,M_2 /M_1)<\dots<\cd(\fa,M_{r-1} /M_{r-2})<\cd(\fa,M/M_{r-1}).$ In the case where, $M$ is finitely generated a sequentially Cohen-Macaulay with respect to $\fa$ is called \textit{finite $\fa$-sequentially Cohen-Macaulay}.}
\end{defn}
Here notice that the concept of relative sequentially Cohen-Macaulay is a generalization of the notion of sequentially Cohen-Macaulay which is defined in the introduction. Indeed, for a finitely generated $R$-module $M$ over a local ring $(R,\fm)$, one has $\cd(\fm, M)=\dim_R(M)$ and also we can see that the concept of $\fm$-relative Cohen-Macaulay modules coincide with the ordinary concept of Cohen-Macaulay modules, and so $\fm$-sequentially Cohen-Macaulay modules are precisely the sequentially Cohen-Macaulay modules. Also, it is clear that every $\fa$-relative Cohen-Macaulay $R$-modules are also $\fa$-relative sequentially Cohen-Macaulay, but in general, the converse is no longer true. For this purpose, let $R=k[[x,y]]$ where $k$ is a field, $M=k[[x,y]]/(xy)$ and $\fa=(x)$. Notice that $\cd(\fa,M)=1$ and $\Gamma_{\fa}(M)\neq 0$; and so $M$ is not $\fa$-relative Cohen-Macaulay. But,  $0\subset\Gamma_{\fa}(M)\subset M$ is a finite filtration such that $\Gamma_{\fa}(M)$ and $M/\Gamma_{\fa}(M)$ are $\fa$-relative Cohen-Macaulay $R$-modules with cohomological dimension $0$ and $1$, respectively. Therefore $M$ is a finite $\fa$-relative sequentially Cohen-Macaulay $R$-module which is not finite $\fa$-relative Cohen-Macaulay.

\section{Relative Sequentially Cohen-Macaulay and relative Canonical modules}

The starting point of this section is the next two lemmas, which will play an essential role in proving our main result \ref{1}.
\begin{lem}\label{Lem}Let $(R,\fm)$ be an $\fa$-relative Cohen-Macaulay local ring with $\cd(\fa,R)=c$ and suppose that $M$ is a finite $\fa$-relative sequentially Cohen-Macaulay module with the filtration $0=M_0 \subset M_1 \subset\dots \subset M_{r-1}\subset M_{r}=M$ and set $\cd(\fa,M_i /M_{i-1}):=c_i$. Then $\Ext^{c-j}_{R}(M,\D_{\fa})=0$ for all $j\notin\{c_1,\dots, c_r\}$ and $\Ext^{c-c_i}_{R}(M,\D_{\fa})$ is $\fa$-relative Cohen-Macaulay of cohomological dimension $c_i$ for all $i=1,\dots, r.$
\begin{proof}
We use induction on $r$ to prove the result. To do this, let  $r=1$, then by our assumption $M$ is finite $\fa$-relative Cohen-Macaulay with $\cd(\fa,M)=c_1$. So, by Theorem \ref{3}, one has $\Ext_R^{c-i}(M,\D_{\fa})=0$ for all $i\neq c_1$ and $\Ext_R^{c-c_1}(M,\D_{\fa})$ is an $\fa$-relative Cohen-Macaulay module with $\cd(\fa,\Ext_R^{c-c_1}(M,\D_{\fa}))=c_1$. Now, assume that $r>1$ and the result has been proved for all finite $\fa$-relative sequentially Cohen-Macaulay modules with a finite filtration of length $r-1$. Notice that $M/M_1$ is a finite $\fa$-relative sequentially Cohen-Macaulay $R$-module with the following finite filtration of length $r-1$:
$$0=M_1/M_1 \subset M_2/M_1 \subset\dots \subset M_{r}/M_1=M/M_1.$$ Therefore, by induction hypothesis one has $\Ext_R^{c-j}(M/M_1,\D_{\fa})=0$ for all $j\notin \{c_2 ,\dots, c_{r}\}$ and $\Ext^{c-c_i}_{R}(M/M_1,\D_{\fa})$ is $\fa$-relative Cohen-Macaulay of cohomological dimension $c_i$ for all $i=2,\dots, r.$
Now consider the short exact sequence $0 \longrightarrow M_1 \longrightarrow M \longrightarrow M/M_1 \longrightarrow0$ and the induced long exact sequence
\[\begin{array}{rl}
\cdots\longrightarrow\Ext_R^{c-i}(M/M_1,\D_{\fa})\longrightarrow\Ext_R^{c-i}(M,\D_{\fa})\longrightarrow \Ext_R^{c-i}(M_1,\D_{\fa})\\ \longrightarrow\Ext_R^{c-i+1}(M/M_1,\D_{\fa})\longrightarrow\Ext_R^{c-i+1}(M,\D_{\fa})\longrightarrow\cdots.
\end{array}\]
Notice that, by Theorem \ref{3}, $\Ext_R^{c-j}(M_1,\D_{\fa})=0$ for all $j\neq c_1$ and $\Ext_R^{c-c_1}(M_1,\D_{\fa})$ is $\fa$-relative Cohen-Macaulay with $\cd(\fa,\Ext_R^{c-c_1}(M_1,\D_{\fa}))=c_1.$ Therefore, by the above induced sequence one has the isomorphisms $\Ext_R^{c-j}(M,\D_{\fa})\cong\Ext_R^{c-j}(M/M_1,\D_{\fa})$ for $j\neq c_1 , c_{1}-1$ and the following exact sequence:
\[\begin{array}{rl}
0\longrightarrow\Ext_R^{c-c_1}(M/M_1,\D_{\fa})\longrightarrow\Ext_R^{c-c_1}(M,\D_{\fa})\longrightarrow\Ext_R^{c-c_1}(M_1,\D_{\fa})\\ \longrightarrow\Ext_R^{c-c_1+1}(M/M_1,\D_{\fa})\longrightarrow\Ext_R^{c-c_1+1}(M,\D_{\fa})\longrightarrow 0.
\end{array}\]
Since $\Ext_R^{c-j}(M/M_1,\D_{\fa})=0$ for $j=c_1, c_1-1$, one has  $\Ext_R^{c-c_1}(M,\D_{\fa})\cong\Ext_R^{c-c_1}(M_1,\D_{\fa})$ and $\Ext_R^{c-c_1+1}(M,\D_{\fa})=0$, as required.
\end{proof}
\end{lem}


\begin{lem}\label{22}Let $\fa$ be an ideal of $R$ with $\cd(\fa,R)=c$, $M$ be a finitely generated $R$-module and $D$ be an $R$-module which the following conditions are established:
\begin{itemize}
\item[(i)]{{$\Hom_R(D,D)$ is flat and $\Ext_R^i(D,D)=0$ for all $i>0$.}}
\item[(ii)]{{For all $i$, we have $\Ext_{R}^{c-i}(M,D)$ is zero, or}}
\begin{itemize}
\item[(a)]{{$\Ext_{R}^{c-j}(\Ext_{R}^{c-i}(M,D),D)=0$ if and only if $j\neq i$, and}}
\item[(b)]\emph{{{$\H_{\fa}^j(\Ext_{R}^{c-i}(\Ext_{R}^{c-i}(M,D),D))=0$ if and only if $j\neq i$.}}}
\end{itemize}
\end{itemize}
 Then $M\otimes_R\Hom_{R}(D,D)$ is $\fa$-sequentially Cohen-Macaulay.
\begin{proof}Let $P_\bullet:\cdots\longrightarrow P_1\longrightarrow P_0\longrightarrow 0$ be a deleted free resolution with finitely generated terms for $M$ and $I_\bullet:0\longrightarrow I^0\longrightarrow I^1\longrightarrow \cdots\longrightarrow I^c\longrightarrow\cdots$ be a deleted injective resolution for $D$. Now, consider the first quadrant double complex $M_{p,q}:=\Hom_R(\Hom_R(P_p,D),I^{c-q})$. Let $^IE$ (resp. $^{II}E$) denote the vertical (resp. horizontal) spectral sequence
associated to the double complex $\mathcal{M}=\{M_{p,q}\}$. Now, with the notation of \cite{JR}, $\E^1$ is the bigraded module whose $(p,q)$ term is $\H^{''}_{q}(M_{p,*})$, the q-th homology of the p-th column. Therefore, $\H^{''}_{q}(M_{p,*})=\Ext_R^{c-q}(\Hom_R(P_p,D),D)$, and so in view of $(i)$ one has
\[ ^{I}E_{p,q}^{1}=\H^{''}_{q}(M_{p,*})=\begin{cases}
       0 & \text{if $q\neq c$}\\
       \Hom_R(\Hom_R(P_p,D),D) & \text{if $q=c$},
       \end{cases} \]
As  $P_p$ is finite free module for all $p$, then in view of \cite[Theorem 2.5.6]{CF1}, one has $\Hom_R(\Hom_R(P_\bullet,D),D)\simeq P_\bullet\otimes_R\Hom_R(D,D)$.
So, we have
       \[ ^{I}E_{p,q}^{2}=\H^{'}_{p}\H^{''}_{q}(\mathcal{M})=\begin{cases}
       0 & \text{if $q\neq c$}\\
       \Tor^{R}_{p}(M, \Hom_R(D,D)) & \text{if $q=c$}.
       \end{cases} \]
       Hence, in view of $(i)$, $^{I}E_{p,q}^{2}=0$ for all $p>0$ and so the spectral sequence collapses on the $q$-axis. Therefore, by \cite[Proposition 10.21]{JR}, one has $^{I}E_{0,c}^{2}=M\otimes_R\Hom_R(D,D)\cong \H_c(\Tot(M))$.

       A similar argument applies to the second iterated homology, using the fact that each $I_{c-q}$ is injective, yields
       $^{II}E_{p,q}^{2}=\H^{''}_{p}\H^{'}_{q}(\mathcal{M})=\Ext_R^{c-p}(\Ext_R^q(M,D),D)$ and that $^{II}E_{p,q}^{2}\underset{p}\Longrightarrow\H_{p+q}(\Tot(\mathcal{M}))$. Thus, for all $p,q$ such that $c=p+q$, there is the following filtration  $${0}=\Phi^{-1}H_{c}\subseteq\Phi^{0}H_{c}\subseteq\ldots\subseteq\Phi^{c-1}H_{c}\subseteq\Phi^{c}H_{c}=H_{c}(\Tot(M))=M\otimes_R\Hom_R(D,D)$$
       such that $E_{p,c-p}^{\infty}:=^{II}E_{p,c-p}^{\infty}\cong\Phi^{p}H_{c}/\Phi^{p-1}H_{c}$.

       On the other hand, by our assumption $(ii)$, one has $\Ext_R^{c-p}(\Ext_R^q(M,D),D)=0$ for all $p\neq c-q$ and so
       $\E^{2}_{p,q}=0$ for all $p\neq c-q$. Next, for each $r\geq2$ and $p\geq 0$, let $$Z^{r}_{p,c-p}:=\Ker(\E^{r}_{p,c-p}\longrightarrow\E^{r}_{p-r,c-p+r-1})$$ and $$B^{r}_{p,c-p}:=
\Im(\E^{r}_{p+r,c-p-r+1}\longrightarrow\E^{r}_{p,c-p}).$$ By using the assumption, $\E^{r}_{p-r,c-p+r-1}=\E^{r}_{p+r,c-p-r+1}=0$ for all $r\geq2$ and $p\geq 0$, because for all $p$, $\E^{r}_{p,c-p}$ is subquotients of $\E^{2}_{p,c-p}$, and so $\E^{r}_{p,c-p}=Z^{r}_{p,c-p}$ and $B^{r}_{p,c-p}=0$. On the other hand, one has $\E^{r+1}_{p,c-p}=\frac{Z^{r}_{p,c-p}}{B^{r}_{p,c-p}}$, and thus $\E^2_{p,c-p}\cong \E^{3}_{p,c-p}\cong\dots\cong\E_{p,c-p}^{\infty}$ for all $p\geq 0$. Therefore, by
       our assumption $(iii)$ for all $p$, $\Ext^{c-p}_{R}(\Ext^{c-p}_{R}(M,D),D)$ is $\fa$-relative Cohen-Macaulay of cohomological dimension $p$. Therefore, $\Phi^{p}H_{c}/\Phi^{p-1}H_{c}$ is $\fa$-relative Cohen-Macaulay of cohomological dimension $p$ for all $0\leq p\leq c$, and so the above filtration is a $\fa$-relative Cohen-Macaulay filtration for $M\otimes_R\Hom_R(D,D)$, as required.

\end{proof}
\end{lem}


\begin{thm}\label{1}Let $(R,\fm)$ be an $\fa$-relative Cohen-Macaulay local ring with $\cd(\fa,R)=c$. Then the following statements hold:
 \begin{itemize}
\item[(i)]{{ If $M$ is finite $\fa$-relative sequentially Cohen-Macaulay, then $\Ext^{c-i}_{R}(M,\D_{\fa})$ is zero or is $\fa$-relative Cohen-Macaulay of cohomological dimension $i$ for all $i\geq 0$.}}
\item[(ii)]{{ If for all $i\geq 0$, $\Ext^{c-i}_{R}(M,\D_{\fa})$ is zero or is $\fa$-relative Cohen-Macaulay of cohomological dimension $i$ and \emph{$\H_{\fa}^j(\Ext^{c-i}_{R}(\Ext^{c-i}_{R}(M,\D_{\fa}),\D_{\fa}))=0$} for all $j>i$, then $M\otimes_R\widehat R$ is $\fa$-relative sequentially Cohen-Macaulay.}}
\end{itemize}
\begin{proof}The statement $(i)$ is an immediate consequence of Lemma \ref{Lem}. For the statement $(ii)$, first in view of  Theorem \ref{3}(iii) one has $\widehat{R}\cong\Hom_{R}(\D_{\fa},\D_{\fa})$ and $\Ext_{R}^i(\D_{\fa},\D_{\fa})=0$ for all $i>0$. Also, by our assumption and \cite[Corollary 3.2]{MZ2}, we have $\Ext_R^{c-j}(\Ext_R^{c-i}(M,\D_{\fa}),\D_{\fa})=0$ if and only if $j\neq i$.

Now, we show that for all $i$, $\H_{\fa}^j(\Ext^{c-i}_{R}(\Ext^{c-i}_{R}(M,\D_{\fa}),\D_{\fa}))=0$ for all $j<i$, which to do this it is enough to show that $\Ext_{R}^j(R/\fa, \Ext^{c-i}_{R}(\Ext^{c-i}_{R}(M,\D_{\fa}),\D_{\fa}))=0$ for all $j<i$. To do this, notice that since $\Ext^{c-i}_{R}(M,\D_{\fa})$ is $\fa$-relative Cohen-Macaulay of cohomological dimension $i$, in view of \cite[Theorem 3.1(iii)]{MZ2}, we have $\Tor_{t-i}^R(R/\fa, \Ext^{c-i}_{R}(M,\D_{\fa}))\cong\Tor_{t}^R(R/\fa,\H_{\fa}^{i}(\Ext^{c-i}_{R}(M,\D_{\fa})))$ for all $t$. Therefore, one can use \cite[Corollary 3.2]{MZ2} to get the following isomorphisms
\[\begin{array}{rl}
\Ext_{R}^j(R/\fa, \Ext^{c-i}_{R}(\Ext^{c-i}_{R}(M,\D_{\fa}),\D_{\fa}))&\cong(\Tor_{j}^R(R/\fa,\H_{\fa}^{i}(\Ext^{c-i}_{R}(M,\D_{\fa}))))^{\vee}\\
&\cong(\Tor_{j-i}^R(R/\fa,\Ext^{c-i}_{R}(M,\D_{\fa})))^{\vee},
\end{array}\]
which implies that $\Ext_{R}^j(R/\fa, \Ext^{c-i}_{R}(\Ext^{c-i}_{R}(M,\D_{\fa}),\D_{\fa}))=0$ for all $j<i$, as required. Therefore, by
       our assumption we can see that $\Ext^{c-i}_{R}(\Ext^{c-i}_{R}(M,\D_{\fa}),\D_{\fa})$ is $\fa$-relative Cohen-Macaulay of cohomological dimension $i$, for all $i$. Therefore, the assertion is done by Lemma \ref{22}.
\end{proof}
\end{thm}
The following result is a consequence of Lemma \ref{22} which has already been proved by J. Herzog and D. Popescu in \cite[Theorem 2.4]{HEP}.

\begin{cor}Let $(R,\fm)$ be a $d$-dimensional Cohen-Macaulay local ring with the canonical $R$-module $\omega_{R}$ and $M$ be a finitely generated $R$-module. Then, the following statements are equivalent:
\begin{itemize}
\item[(i)]{{$M$ is sequentially Cohen-Macaulay.}}
\item[(ii)]{{$\Ext^{d-i}_{R}(M,\omega_{R})$ is zero or Cohen-Macaulay of dimension $i$ for all $i\geq 0.$}}
\end{itemize}
\begin{proof}
First notice that for all finitely generated $R$-modules $X$ we have $\dim_R(X)=\dim_{\widehat R}(X\otimes_R \widehat R)$, $\depth_R(X)=\depth_{\widehat R}(X\otimes_R \widehat R)$ and $\id_R(X)=\id_{\widehat R}(X\otimes_R \widehat R)=\depth R$. On the other hand, , it is clear to see that
\[\begin{array}{rl}
(1)~~~\omega_{\widehat{R}}&\cong\omega_R\otimes_R{\widehat{R}}\\
&\cong\Hom_R(\H_{\fm}^{d}(R),\E_R(R/\fm))\\
&\cong\Hom_{\widehat R}(\H_{\fm}^{d}(R)\otimes_R\widehat R,\E_R(R/\fm))\\
&\cong\Hom_{\widehat R}(\H_{\fm\widehat R}^{d}(\widehat R),\E_{\widehat R}(\widehat{R}/\fm\widehat R))\\
&=\D_{{\fm}\widehat R}.
\end{array}\]
Hence, for all integers $i$, one has
\[\begin{array}{rl}
(2)~~~\Ext^{d-i}_{\widehat{R}}(M\otimes_R{\widehat{R}},\D_{{\fm}\widehat R })&\cong\Ext^{d-i}_{\widehat R}(M\otimes_R \widehat R,\omega_{\widehat R})\\
&\cong\Ext^{d-i}_{R}(M,\omega_{R})\otimes_R \widehat R.
\end{array}\]
Therefore, for the implication $(ii)\Longrightarrow (i)$, in view of Theorem \ref{3} and the fact that $\widehat R$ is faithfully flat, one can deduce that
$\id_{R}(\omega_R)=\dim R$, $R\cong\Hom_R(\omega_{R},\omega_{R})$ and $\Ext_{R}^i(\omega_R,\omega_R)=0$ for all $i>0$. Also, by our assumption and Theorem \ref{1}(iv), for all $i= 0, 1,\ldots,\dim R $, the $i$-dimensional finitely generated Cohen-Macaulay $R$-modules $\Ext_{R}^{d-i}(M,\omega_{R})$, have the following conditions:
\begin{itemize}
\item[(a)]{{$\Ext_{R}^{d-j}(\Ext_{R}^{d-i}(M,\omega_{R}),\omega_{R})=0$ if and only if $j\neq i$.}}
\item[(b)]{{$\H_{\fa}^j(\Ext_{R}^{d-i}(\Ext_{R}^{d-i}(M,\omega_{R}),\omega_R))=0$ if and only if $j\neq i$.}}
\end{itemize}
Therefore, one can use Lemma \ref{22} to complete the proof. Note that, by use of definition, one can check that $M\otimes_R \widehat R$ is a sequentially Cohen-Macaulay $\widehat R$-module whenever $M$ is a sequentially Cohen-Macaulay $R$-module. Hence, the implication $(i)\Longrightarrow (ii)$ follows from Theorem \ref{1}(i) and the above two isomorphisms $(1)$ and $(2)$.
\end{proof}
\end{cor}
\section{Relative sequentially Cohen-Macaulay and local homology}
Based on the discussion made in the introduction, an Artinain $R$-module $M$ is co-Cohen-Macaulay if and only if $\H^{\fm}_i(M)=0$ for all $i\neq \Nd M$.
Hence, in this section, as a generalization of this concept, our main aim is to study the $R$-modules $X$ such that its local homology modules concerning an arbitrary ideal $\fa$, $\H_{i}(\Lam_{\fa}(X))$, has just one non-vanishing point. Indeed, in this section, we are going to provide some characterizations for relative Cohen-Macaulay and relative sequentially Cohen-Macaulay modules.

The starting point of this section is the next two lemmas which play an essential role in the proof of the next main theorem.
\begin{lem}\label{00}Let $X_{\bullet}:=\cdots \longrightarrow X_{l+1}\stackrel{d_{l+1}}\longrightarrow X_{l} \stackrel{d_{l}}\longrightarrow X_{l-1\longrightarrow\cdots }$ be an $R$-complex such that $H_{i}(X_{\bullet})=0$ for all $i\neq c$, then there exists the isomorphism $\Sigma^c\H_{c}(X_{\bullet})\simeq X_{\bullet}$ in derived category $\D(R)$.
\end{lem}
\begin{proof}Consider the following quasi-isomorphisms:
 $$\begin{CD}
\Sigma^{c}\H_c(X_{\bullet}):\cdots @>>>0 @>>> 0 @>>> \H_c(X_{\bullet}) @>>>0 @>>>0@>>>\cdots\\
@. @V0VV @VV0V @VViV@V0VV@V0VV \\
\subset_c X_{\bullet}:\cdots @>>> 0 @>>>0 @>>>\frac{X_{c}}{\Im d_{c+1}} @>\overline{d_c}>>X_{c-1}@>{d_{c-1}}>>X_{c-2}@>>> \cdots \\
@. @A0AA @A0AA @A\pi AA@|@| \\
X_{\bullet}:\cdots @>>> X_{c+2} @>{d_{c+2}}>>X_{c+1} @>{d_{c+1}}>> X_{c} @>{d_c}>>X_{c-1}@>{d_{c-1}}>>X_{c-2}@>>> \cdots ,
\end{CD}$$ and

 $$\begin{CD}
X_{\bullet}:\cdots @>>> X_{c+2} @>{d_{c+2}}>>X_{c+1} @>{d_{c+1}}>> X_{c} @>{d_c}>>X_{c-1}@>>> \cdots\\
@. @| @| @AiAA@A0AA\\
{X_{\bullet}}_{c}\supset:\cdots@>>> X_{c+2}@>d_{c+2}>>X_{c+1} @>iod_{c+1}>>\ker d_{c}@>>>0@>>>\cdots\\
@. @V0VV @VVV @V\pi VV @VV0V\\
\Sigma^{c}\H_{c}(X_{\bullet}):\cdots @>>>0 @>>>0@>>>\H_{c}(X_{\bullet}) @>>>0 @>>>\cdots\\
\end{CD}$$
Therefore, there exists the quasi-isomorphisms
\[\begin{cases}
\Sigma^{c}\H_{c}(X_{\bullet})\stackrel{\simeq}\longrightarrow\subset_c X_{\bullet}\stackrel{\simeq}\longleftarrow X_{\bullet}&\text{and}\\
\Sigma^{c}\H_{c}(X_{\bullet})\stackrel{\simeq}\longleftarrow {X_{\bullet}} _c\supset\stackrel{\simeq}\longrightarrow X_{\bullet}& ,
\end{cases}\]
where $\subset_cX_{\bullet}$ and ${X_{\bullet}} _c\supset$ are the soft left truncation and the soft right truncation
of the complex $X_{\bullet}$, respectively, for more detail see \cite[A.1.14]{CF0}. Therefore, by the definition of isomorphism in derived category $\D(R)$ one has $\Sigma^{-c}\H^{c}_{\fa}(M)$ and $\Gamma_{\fa}(I^\bullet)$ are isomorphic in derived category modules $\D(R)$.
\end{proof}

\begin{lem}\label{000}Let $M$ be an $R$-module, $\fa$ be an ideal of $R$ such that for a fix integer $c$, \emph{$\H_{\fa}^c(M)\neq 0$} and \emph{$\H_{i}(\Lam_{\fa}(\H_{\fa}^c(M)))=0$} for all $i\neq c$. Then, \emph{$\H_{c}(\Lam_{\fa}(\H_{\fa}^c(M)))$} is a nonzero $\fa$-relative Cohen-Macaulay $R$-module of cohomological dimension $c$.
\end{lem}
\begin{proof}First notice that by our assumption and Lemma \ref{00} we have the isomorphism $\Sigma^c\H_{c}(\Lam_{\fa}(\H_{\fa}^c(M)))\simeq\Lam_{\fa}(\H_{\fa}^c(M))$, and so $\H_{c}(\Lam_{\fa}(\H_{\fa}^c(M)))\simeq\Sigma^{-c}\Lam_{\fa}(\H_{\fa}^c(M)).$
Next, consider the following isomorphisms in derived category $\D(R)$:
\[\begin{array}{rl}
\ugamma_{\fa}(\H_{c}(\Lam_{\fa}(\H_{\fa}^c(M))))&\simeq\ugamma_{\fa}(\Sigma^{-c}\Lam_{\fa}(\H_{\fa}^c(M)))\\
&\simeq\check{C}_{\underline{x}}\utp_{R}\Sigma^{-c}(\Lam_{\fa}(\H_{\fa}^c(M)))\\
&\simeq\Sigma^{-c}(\check{C}_{\underline{x}}\utp_{R}\Lam_{\fa}(\H_{\fa}^c(M)))\\
&\simeq\Sigma^{-c}\ugamma_{\fa}(\Lam_{\fa}(\H_{\fa}^c(M)))\\
&\simeq\Sigma^{-c}\ugamma_{\fa}(\H_{\fa}^c(M)),
\end{array}\]
where the second and last isomorphisms follow from \cite[Theorem 7.4.4]{SC2} and \cite[Theorem 9.1.3]{SC2}, respectively. Therefore, for all $i$, one has the following isomorphisms:
\[\begin{array}{rl}
\H^{i}_{\fa}(\H_{c}(\Lam_{\fa}(\H_{\fa}^c(M))))&\cong \H_{-i}(\ugamma_{\fa}(\H_{c}(\Lam_{\fa}(\H_{\fa}^c(M))))\\
&\cong\H_{-i}(\Sigma^{-c}\ugamma_{\fa}(\H_{\fa}^c(M)))\\
&\cong\H_{-i+c}(\ugamma_{\fa}(\H_{\fa}^c(M)))\\
&\cong\H^{i-c}_{\fa}(\H_{\fa}^c(M)),
\end{array}\]
which show that $\H^{i}_{\fa}(\H_{c}(\Lam_{\fa}(\H_{\fa}^c(M))))=0$ if and only if $i\neq c$, as required.
\end{proof}

The following theorem, which is one of our main results, provides a characterization of a finite $\fa$-relative Cohen-Macaulay $R$-module $M$ in terms of vanishing of the local homology modules of the local cohomology module $\H_{\fa}^{\cd(\fa, M)}(M)$. Furthermore, it shows that the converse of \cite[Theorem 1.2]{SC1} is also true.
\begin{thm}\label{2}Let $\fa$ be an ideal of $R$ and $M$ be an $R$-module with $\cd(\fa,M)=c$. Then, the following statements hold:
\begin{itemize}
\item[(i)]{{ If $M$ is $\fa$-relative Cohen-Macaulay, then there exists the following isomorphism: \emph{$$\underset{n\in\mathbb{N}}\varprojlim\Tor^R_{i+c}(R/\fa^n, \H^{c}_{\fa}(M))\cong\underset{n\in\mathbb{N}}\varprojlim\Tor^R_{i}(R/\fa^n, M),$$} for all integers $i$. In particular, we have \emph{$\H^{\fa}_{j+c}(\H_{\fa}^c(M))\cong\H^{\fa}_j(M)$} for all $j$.}}
\item[(ii)]{{ If $M$ is finite $\fa$-relative Cohen-Macaulay, then \emph{\[\H^{\fa}_i(\H_{\fa}^c(M))\cong\begin{cases}
       \widehat{M}^{\fa} & \text{if $i=c$}\\
       0& \text{otherwise},
       \end{cases} \]}}}
       \item[(iii)]{{ $M$ is finite $\fa$-relative Cohen-Macaulay if and only if \emph{\[\H_{i}(\Lam_{\fa}(\H_{\fa}^c(M))\cong\begin{cases}
       \widehat{M}^{\fa} & \text{if $i=c$}\\
       0& \text{otherwise},
       \end{cases} \]}}}
\end{itemize}

\begin{proof}$(i)$. Let
$$I^{\bullet}: 0 \longrightarrow I_{0}\stackrel{d_0}\longrightarrow I_{-1}\stackrel{d_{-1}}\longrightarrow \cdots\longrightarrow I_{-c+1}\stackrel{d_{-c+1}}\longrightarrow I_{-c}\stackrel{d_{-c}}\longrightarrow I_{-c-1} \longrightarrow\cdots$$
be a deleted injective resolution for $M$. By applying the functor $\Gamma_{\fa}(-)$ on this complex, we obtain the complex $\Gamma_{\fa}(I^{\bullet})$. As, $\H_{\fa}^i(M)=\H_{-i}(\Gamma_{\fa}(I^{\bullet}))$, by our assumption and Lemma \ref{00} we have $\Sigma^{-c}\H^{c}_{\fa}(M)$ and $\Gamma_{\fa}(I^\bullet)$ are isomorphic in derived category modules $\D(R)$. Now, as $\Gamma_{\fa}(I^\bullet)$ is a bounded above complex of injective modules, by \cite[Lemma 3.3.8]{CF1} one has $\Gamma_{\fa}(I^\bullet)$ is a semi-injective complex. So, in view of \cite[Proposition 4.1.11]{CF1} there exists a quasi-isomorphism $f:\Sigma^{-c}\H^{c}_{\fa}(M)\stackrel{\simeq}\longrightarrow\Gamma_{\fa}(I^\bullet)$. Then, we have the following functorial morphisms in the derived category $\D(R)$:
\[\begin{array}{rl}
(1)~~-\otimes_R \Sigma^{-c}\H^{c}_{\fa}(M)&\stackrel{\Id(-)\otimes_R f}\longrightarrow -\otimes_R \Gamma_{\fa}(I^\bullet)\\
&\stackrel{\psi(-,I^\bullet)}\longrightarrow\ugamma_{\fa}(-\otimes_R I^\bullet)\\
&\stackrel{i(-\otimes_R I^\bullet)}\longrightarrow -\otimes_R I^\bullet\\
\end{array}\]
Here, we should notice that for any semi-projective complex $P_{\bullet}$, the functorial morphisms $i_{P_\bullet}\otimes_R f$,  by \cite[Theorem 3.4.5]{CF1}, is quasi-isomorphism and $\psi(-,I^\bullet)$ follows from \cite[Corollary 3.3.1]{LIP} which is always an isomorphism. Also, the morphism $i(-\otimes_R I^\bullet)$ follows from \cite[Corollary 3.2.1]{LIP} which is isomorphism for all $R$-complex $C$ such that $\Supp_R(C)\subseteq\V(\fa)$. Therefore, there is a functorial isomorphism $f_{P_\bullet}:P_\bullet\otimes_R \Sigma^{-c}\H^{c}_{\fa}(M)\longrightarrow P_\bullet\otimes_R I^\bullet$ in Derided category $\D(R)$ for each deleted projective resolution  $P_\bullet$ of an $\fa$-torsion $R$-module $N$. Now, let $n<m$ be natural numbers and $P_{\bullet}^m$ and $P_{\bullet}^n$ be deleted projective resolutions for $R/\fa^m$ and $R/\fa^n$, respectively. Consider the natural map $\gamma_n^m: R/\fa^m\longrightarrow R/\fa^n$ which induces the $R$-complex morphism $({\gamma}_n^m)_{i\in \mathbb{N}}: P_{\bullet}^m\longrightarrow P_{\bullet}^n$. Therefore, one can obtain the following commutative diagrams
$$\begin{CD}
P_\bullet^m\otimes_R \Sigma^{-c}\H^{c}_{\fa}(M)@>f_{P_\bullet^m}>> P_\bullet^m\otimes_R I^{\bullet}\\
@V({\gamma}_n^m)_{i\in \mathbb{N}}\otimes_R \Id_{\Sigma^{-c}{H^{c}(M)}}VV @V({\gamma}_n^m)_{i\in \mathbb{N}}\otimes_R Id_{I^\bullet}VV \\
P_\bullet^n\otimes_R \Sigma^{-c}\H^{c}_{\fa}(M)@>f_{P_\bullet^n}>> P_\bullet^n\otimes_R I^{\bullet}\\
\end{CD}$$
Now, taking homology of the above commutative diagram yields the following commutative diagram:
$$\begin{CD}
\Tor^R_{i+c}(R/\fa^m, \H^{c}_{\fa}(M))@>\dagger>>\Tor^R_{i}(R/\fa^m,M)\\
@VVV @VVV \\
\Tor^R_{i+c}(R/\fa^n, \H^{c}_{\fa}(M))@>\ddagger>>\Tor^R_{i}(R/\fa^n,M),\\
\end{CD}$$
for all integers $i$, where $\dagger$ and $\ddagger$ are isomorphisms. Here, one should notice that $\H_i(P_\bullet^m\otimes_R \Sigma^{-c}\H^{c}_{\fa}(M))=\H_i(\Sigma^{-c}(P_\bullet^m\otimes_R \H^{c}_{\fa}(M)))=\H_{i+c}(P_\bullet^m\otimes_R \H^{c}_{\fa}(M))=\Tor^R_{i+c}(R/\fa^m, \H^{c}_{\fa}(M)),$ and also $\H_i(P_\bullet^m\otimes_R I^{\bullet})=\Tor^R_{i}(R/\fa^m,M).$ Therefore, one can deduce that $\underset{n\in\mathbb{N}}\varprojlim\Tor^R_{i+c}(R/\fa^n, \H^{c}_{\fa}(M))\cong\underset{n\in\mathbb{N}}\varprojlim\Tor^R_{i}(R/\fa^n, M)$ and thus $\H^{\fa}_{i+c}(\H_{\fa}^c(M))\cong\H_{i}^{\fa}(M)$, as required.

$(ii).$ It follows from the isomorphism in $(i)$ and \cite[Remark 3.2]{CN}(ii).

 $(iii)$. For the implication $\Longrightarrow$, first in view of Lemma \ref{00} we have the isomorphism $\Sigma^{-c}\H_{\fa}^c(M)\simeq\ugamma_{\fa}(M)$. On the other hand, by \cite[Theorem 9.1.3]{SC2} there is the isomorphism $\Lam_{\fa}(\ugamma_{\fa}(M)\simeq\Lam_{\fa}(M)$ in derived category $\D(R)$, and also in view of \cite[Theorem 7.5.12]{SC2} one has $\Lam_{\fa}(-)\simeq\uhom_R(\check{C}_{\underline{x}}, -)$. Therefore, one can get the following isomorphisms in the derived category $\D(R)$:
 \[\begin{array}{rl}
\H_{i}(\Lam_{\fa}(M))&\cong\H_{i}(\Lam_{\fa}(\ugamma_{\fa}(M))\\
&\cong\H_{i}(\Lam_{\fa}(\Sigma^{-c}\H_{\fa}^c(M))\\
&\cong\H_{i}(\uhom_R(\check{C}_{\underline{x}},\Sigma^{-c}\H_{\fa}^c(M)))\\
&\cong\H_{i}(\Sigma^{-c}\uhom_R(\check{C}_{\underline{x}},\H_{\fa}^c(M)))\\
&\cong\H_{i+c}(\uhom_R(\check{C}_{\underline{x}},\H_{\fa}^c(M)))\\
&\cong\H_{i+c}(\Lam_{\fa}(\H_{\fa}^c(M))).
\end{array}\]
Hence, by \cite[Remark 2.1]{CN}(iii) the assertion is done. For the implication $\Longleftarrow$, by Lemma \ref{000} and our assumption one has $\H_{\fa}^i(\widehat{M}^{\fa})=0$ for all $i\neq c$.
On the other hand, by the Flat Base Change Theorem and Independent Theorem we have the following isomorphisms for all integers $i$:
 \[\begin{array}{rl}
 \H_{\fa}^{i}(\widehat{M}^{\fa})&\cong\H_{\fa\widehat{R}^{\fa}}^{i}(M\otimes_R\widehat{R}^{\fa})\\
 &\cong\H_{\fa}^{i}(M)\otimes_R\widehat{R}^{\fa}\\
 &\cong\H_{\fa}^{i}(M).
\end{array}\]
Here we should notice the last isomorphism holds because the $R$-module $\H_{\fa}^{i}(M)$ is $\fa$-torsion, and so has $\widehat R^{\fa}$-module structure and also, by \cite[Lemma 1.4]{KLS}, there exists the isomorphism  $\H_{\fa}^{i}(M)\otimes_R\widehat{R}^{\fa}\cong\H_{\fa}^{i}(M)$ both as $R$-modules and $\widehat R^{\fa}$-modules. Therefore, we can deduce that $\H_{\fa}^i(M)=0$ for all $i\neq c$ which says that $M$ is $\fa$-relative Cohen-Macaulay.
\end{proof}
\end{thm}

In \cite[Proposition 4.1]{CN} it is shown that for any Artinian $R$-module $M$ there is an isomorphism $\H^{\fa}_i(M)\cong\H_{i}(\Lam_{\fa}(M))$ for all $i$. Also we know that $\H_{\fm}^i(M)$ is Artinian for all $i$ and all finitely generated $R$-modules $M$. So, the following result is an immediate consequence of the previous theorem.

\begin{cor}\label{Cor}Let $(R,\fm)$ be a local ring and $M$ be a finitely generated $R$-module of dimension $d$. Then the following statements are equivalent:
\begin{itemize}
\item[(i)]{{ $M$ is Cohen-Macaulay.}}
\item[(ii)]{{ \emph{\[\H^{\fm}_i(\H_{\fm}^d(M))\cong\begin{cases}
       \widehat{M}^{\fm} & \text{if $i=d$}\\
       0& \text{otherwise}.
       \end{cases} \]
}}}
\end{itemize}
\end{cor}

Notice that the local homology and cohomology functors are $R$-linear. So, the following result follows from Theorem \ref{2}.
\begin{cor} Let $\fa$ be an ideal of $R$ and $M$ be a finitely generated $\fa$-relative Cohen-Macaulay $R$-module with $\cd(\fa,M)=c$. Then
\emph{$\Ann_R(\H_{\fa}^c(M))=\Ann_R(M).$}
\end{cor}

The following lemma has an essential role in the proof of the next theorem, which is one of the main results in the present paper.
\begin{lem}\label{Lem2}Let $M$ is a finite $\fa$-relative sequentially Cohen-Macaulay module with the filtration $0=M_0 \subset M_1 \subset\dots \subset M_{r-1}\subset M_{r}=M$ and set $\cd(\fa,M_i /M_{i-1}):=c_i$. Set \emph{$\T^{i}(-):=\H^{\fa}_i(-)$} or \emph{$\T^{i}(-):=\H_{i}(\Lam_{\fa}(-))$}. Then
\begin{itemize}
\item[(i)]{{\emph{$\H_{\fa}^i(M)=0$} for all $i\notin\{c_1,\dots c_r\}.$}}
\item[(ii)]{{\emph{$\T^j(\H_{\fa}^{i}(M))=0$} for all $j\neq i$ where $i\in\{c_1,\dots, c_r\}.$}}
\item[(iii)]{{\emph{$\T^{c_i}(\H_{\fa}^{c_i}(M))$} is $\fa$-relative Cohen-Macaulay of cohomological dimension $c_i$ for all $i=1,\dots, r.$}}
\end{itemize}
\begin{proof}
We use induction on $r$ to prove the result. To do this, let  $r=1$, then by our assumption $M$ is finite $\fa$-relative Cohen-Macaulay with $\cd(\fa,M)=c_1$. So, by Theorem \ref{2}, one has $\T^j(\H_{\fa}^{c_1}(M))=0$ for all $j\neq c_1$ and $\T^{c_1}(\H_{\fa}^{c_1}(M))\cong\widehat{M}^{\fa}$ which implies that is an $\fa$-relative Cohen-Macaulay module with cohomological dimension $c_1$. Now, assume that $r>1$ and the result has been proved for all finite $\fa$-relative sequentially Cohen-Macaulay modules with a finite filtration of length $r-1$. Notice that $M/M_1$ is a sequentially $R$-module with the following finite filtration of length $r-1$:
$$0=M_1/M_1 \subset M_2/M_1 \subset\dots \subset M_{r}/M_1=M/M_1.$$ Therefore, by induction hypothesis we have $\H_{\fa}^i(M/M_1)=0$ for all $i\notin\{c_2,\dots c_r\},$ $\T^j(\H_{\fa}^{i}(M/M_1))=0$ for all $j\neq i$ where $i\in\{c_2,\dots, c_r\}$ and $\T^{c_i}(\H_{\fa}^{c_i}(M/M_1))$ is $\fa$-relative Cohen-Macaulay of cohomological dimension $c_i$ for all $i=2,\dots, r.$
Now, consider the short exact sequence $0 \longrightarrow M_1 \longrightarrow M \longrightarrow M/M_1 \longrightarrow0$ and the induced long exact sequence
\[\begin{array}{rl}
\cdots\longrightarrow\H_{\fa}^i(M_1)\longrightarrow\H_{\fa}^i(M)\longrightarrow \H_{\fa}^i(M/M_1) \longrightarrow\H_{\fa}^{i+1}(M_1)\longrightarrow\H_{\fa}^{i+1}(M)\longrightarrow\cdots
\end{array}\]
Therefore, by the above induced sequence one has the isomorphism $\H_{\fa}^i(M)\cong\H_{\fa}^i(M/M_1)$ for $i\neq c_1 , c_{1}-1$ and the following exact sequence
\[\begin{array}{rl}
0\longrightarrow\H_{\fa}^{c_1 -1}(M)\longrightarrow\H_{\fa}^{c_1 -1}(M/M_1)\longrightarrow \H_{\fa}^{c_1}(M_1) \longrightarrow\H_{\fa}^{c_1}(M)\longrightarrow\H_{\fa}^{c_1}(M/M_1)\longrightarrow 0.
\end{array}\]
As $\H_{\fa}^{i}(M/M_1)=0$ for $i=c_1, c_1-1$, one has  $\H_{\fa}^{c_1}(M_1)\cong\H_{\fa}^{c_1}(M)$ and $\H_{\fa}^{c_1 -1}(M)=0$. Therefore, using the hypothesis assumption and Theorem \ref{2} for the $R$-module $M_1$ completes the proof.
\end{proof}
\end{lem}

\begin{thm}\label{fff} Let $\fa$ be an ideal of $R$ and $M$ be a finitely generated $R$-module and set \emph{$\T^{i}(-):=\H^{\fa}_i(-)$} or \emph{$\T^{i}(-):=\H_{i}(\Lam_{\fa}(-))$}. Then, the following statements hold:
\begin{itemize}
\item[(i)]{{If $M$ is a finite $\fa$-relative sequentially Cohen-Macaulay, then \emph{$\T^j(\H_{\fa}^i(M))=0$} for all $0\leq i\leq \cd(\fa,M)$ and for all $j\neq i$ and \emph{$\T^i(\H_{\fa}^i(M))$} is zero or $\fa$-relative Cohen-Macaulay of cohomological dimension $i$.}}
\item[(ii)]{{ If \emph{$\H_{j}(\Lam_{\fa}(\H_{\fa}^i(M)))=0$} for all $0\leq i\leq \cd(\fa,M)$ and for all $j\neq i$, then $\widehat{M}^{\fa}$ is a finite $\fa$-relative sequentially Cohen-Macaulay as an $\widehat R^{\fa}$-module.}}
\end{itemize}
\begin{proof}
The statement $(i)$ immediately follows from Lemma \ref{Lem2}. For the statement $(ii)$, we first prove that it will be enough to show the claim for $M =\widehat M^{\fa}$ the
$\fa$-adic completion of $M$ over the completed ring $\widehat {R}^{\fa}$. This follows because, in view of \cite[Corollary 9.8.2]{SC2}, one has $\H_{j}(\Lam_{\fa}(\H_{\fa}^i(M)))\cong\H_{j}(\Lam_{\fa}(\H_{\fa}^i(M)\otimes_R\widehat{R}^{\fa}))\cong\H_{j}(\Lam_{\fa\widehat{R}^{\fa}}(\H_{\fa{\widehat {R}}^{\fa}}^i(\widehat{M}^{\fa})))$. Now, consider $\check{L}_x$ and $E^{\bullet}(M)$ as a bounded free resolution of the Cech complex $\check{C}_x$ and a deleted injective resolution for $M$, respectively. Let
$$\check{L}_x:= 0\longrightarrow \check{L}_{x}^0\longrightarrow \check{L}_{x}^1\longrightarrow\dots \check{L}_{x}^{k-1}\longrightarrow\check{L}_{x}^k\longrightarrow 0,$$ and
  $$\Gamma_{\fa}(E_{\bullet}(M)):= \dots\longleftarrow \Gamma_{\fa}(E_{-t})\longleftarrow\dots \Gamma_{\fa}(E_{-1})\longleftarrow\Gamma_{\fa}(E_{0})\longleftarrow 0,$$
where $k$ is the number of generators of $\fa$ and for all $i\geq 0$, $\Gamma_{\fa}(E_{-i})=\Gamma_{\fa}(E^{i})$. Now, consider the forth quadrant double complex $\mathcal {M}=(M_{p,-q}={\Hom_R(\check{L}_{x}^p,\Gamma_{\fa}(E_{-q}))})$ for all $q\geq 0$ and $0\leq p\leq k$ with the total complex $(\Tot(M)_n=\underset{p-q=n, q\geq 0 , 0\leq p\leq k} \bigoplus M_{p,-q})_{n\leq k}$. Next, consider the first filtration defined by $^{I}F^p\Tot(\mathcal{M}):=\underset{i\leq p}\bigoplus M_{i,i-n}$ and note that since over line $p+q=n$ there are finitely non zero terms of double complex, then this filtration is bounded for all $n$.  Now, with the notation of \cite{JR}, $\E^1$ is the bigraded module whose $(p,-q)$ term is $\H^{''}_{-q}(M_{p,*})$, the -q-th homology of the p-th column. Therefore, $\H^{''}_{-q}(M_{p,*})=\Hom_R({L}_{x}^p,H^{q}_{\fa}(M))$, and so $^{I}E^2_{p,-q}=\H^{'}_p(\H^{''}_{-q}(\mathcal{M}))=\H_{p}(\Lam_{\fa}(\H^{q}_{\fa}(M))$. Now, as only a finite number of rows $M_{\ast,-q}$ are non-zero, then
the first spectral sequence converges, that is,
$$^{I}E^2_{p, -q} = \H_{p}(\Lam_{\fa}(H^{q}_{\fa}(M))\Longrightarrow \H_{p-q}(\Tot(\mathcal M))$$

Therefore, for all $q\geq 0$ and $0\leq p\leq k$ there is the following filtration in the stage $p-q=0$ $${0}=\Phi_{-1}\H^{0}\subseteq\Phi_{0}\H^{0}\subseteq\ldots\subseteq\Phi_{k-1}\H^{0}\subseteq\Phi_{k}H^{0}=\H^{0}(\Tot(M))$$
such that $^{I}E^{\infty}_{p,-q}\cong\Phi_{p}\H^{0}/\Phi_{p-1}\H^{0}$.

Also, by our assumption we have $^{I}E^{\infty}_{p,-q}=0$ for all $p\neq q$ and $^{I}E^{2}_{p,-p}\cong {^{I}E^{\infty}}_{p,-p}=\Phi_{p}\H^{0}/\Phi_{p-1}\H^{0}$ for all $p$. Therefore, by our assumption $\Phi_{p}\H^{0}/\Phi_{p-1}\H^{0}$ is zero or $\fa$-relative Cohen-Macaulay $R$-module of cohomological dimension $p$ and so $\H^{0}(\Tot(M))$ is $\fa$-relative sequentially Cohen-Macaulay. On the other hand, first notice that $\Supp_R(\check{L}_x)\subseteq \V(\fa)$ and so by \cite[Lemma 2.3]{SC1} one has the $R$-complex isomorphism $\Hom_R(\check{L}_x, \Gamma_{\fa}(E^{\bullet}(M)))\simeq\Hom_R(\check{L}_x, E^{\bullet}(M))$. Also, as $E^{\bullet}(M)$ is an injective resolution for $M$, there exists the quasi isomorphism $M\stackrel{\sim}\longrightarrow E^{\bullet}(M)$, and hence we have the quasi isomorphism $\Hom_R(\check{L}_x, M)\stackrel{\sim}\longrightarrow\Hom_R(\check{L}_x, E^{\bullet}(M))$. Therefore, one can use the following isomorphisms

\[\begin{array}{rl}
\H_{0}(\Tot(\mathcal M))&\cong H^{0}(\Hom_R(\check{L}_x, \Gamma_{\fa}(E^{\bullet}(M))))\\
&\cong \H_{0}(\Hom_R(\check{L}_x,M)\\
&\cong\H_0(\Lam_{\fa}(M))\\
&\cong\widehat{M}^{\fa},
\end{array}\]
to deduce that $\widehat M^{\fa}$ is $\fa$-relative sequentially Cohen-Macaulay, as required.

\end{proof}
\end{thm}
Here we should notice that if $c$ is a fix integer such that $\H_{i}(\Lam_{\fa}(\H_{\fa}^c(M)))=0$ for all $i$, then in view of \cite[Corollary 3.4.2]{SC2} one has $\Tor^R_{i}(R/\fa,\H_{\fa}^c(M))=0$ for all $i$. Therefore, by \cite[Theorem 3.2.6]{SC2} we have that $\Ext_R^{i}(R/\fa,\H_{\fa}^c(M))=0$ for all $i$, which implies that $\H_{\fa}^i(\H_{\fa}^c(M))=0$ for all $i$. Hence $\H_{\fa}^c(M)=0$. So, in the statement $(i)$ of the previous theorem, there exists at least one integer $i$ such that $\T^i(\H_{\fa}^i(M))$ is nonzero.

The following result is an immediate consequence of Theorem \ref{fff} for the case where $\fa=\fm$.
\begin{cor}\label{*1}Let $(R,\fm)$ be a local ring and $M$ be a finitely generated $R$-module. Then, the following statements hold:
\begin{itemize}
\item[(i)]{{If $M$ is sequentially Cohen-Macaulay, then \emph{$\H_{j}(\Lam_{\fm}(\H_{\fm}^t(M)))=0$} for all $0\leq t\leq \dim M$ and for all $j\neq t$.}}
\item[(ii)]{{ If \emph{$\H_{j}(\Lam_{\fm}(\H_{\fm}^t(M)))=0$} for all $0\leq t\leq \dim M$ and for all $j\neq t$, then $\widehat M$ is a sequentially Cohen-Macaulay $\widehat R$-module.}}
\end{itemize}
\end{cor}
\begin{defrem}\emph{\label{**}Let $(R,\fm)$ be a local ring of dimension $d$ and $M$ be a finitely generated $R$-module with $n = \dim_R M$. For an integer $i\in\mathbb{Z}$, define $K_M^i:= \H^{i}(\Hom_R(M,D_R^{\bullet})$
where $D_R^{\bullet}$ is a normalaized dualizing complex for $R$.
The module $K_M:= K_{M}^n$ is called the canonical module of $M$. For $i\neq n$ the modules
$K_M^i$ are called the modules of deficiency of $M$. By the local duality theorem, for all $i$ there are the following canonical isomorphisms:
$$\H_{\fm}^i(M)\cong\Hom_R(K_M^i, \E_R(R/ \fm)).$$
Recall that all of the $K_M^i$, $i\in\mathbb{Z}$, are finitely generated $R$-modules. Moreover $M$ is a Cohen-Macaulay
module if and only if $K_M^i = 0$ for all $i\neq n$. Whence the modules of deficiencies of $M$
measure the deviation of $M$ from being a Cohen-Macaulay module. In the case of $(R,\fm)$
the quotient of a local Gorenstein ring $(S, \fn)$ there are the following isomorphisms
$K_M^i=\Ext^{s-i}_{S}(M, S), s = \dim S$, for all $i\in\mathbb{Z}$. Therefore, one has $K_{\widehat M}^i\cong\Hom_R(\H^i_{\fm}(M), \E_R(R/ \fm))\cong\Hom_{\widehat R}(\H^i_{\fm\widehat R}(\widehat M), \E_{\widehat R}(\widehat R/ \fm\widehat R))$ and is finitely generated $\widehat R$-module for all $i$.}
\end{defrem}
\begin{prop}\label{*}Let $(R,\fm)$ be a local ring with a dualizing complex and $M$ be a finitely generated $R$-module. Then, for any integer $t$, the following statements are equivalent:
\begin{itemize}
\item[(i)]{{$K_M^t$ is zero or a Cohen-Macaulay $R$-module of dimension $t$.}}
\item[(ii)]{{ \emph{$\H_{j}(\Lam_{\fm}(\H_{\fm}^t(M)))=0$} for all $j\neq t$.}}
\end{itemize}
\begin{proof} First, consider the following isomorphisms:
\[\begin{array}{rl}
\H_{j}(\Lam_{\fm}(\H_{\fm}^t(M)))&\cong\H_j(\Hom_R(\check{L}_x, \H_{\fm}^t(M))\\
&\cong \H_j(\Hom_R(\check{L}_x, \Hom_R(K_M^t, \E_R(R/ \fm))))\\
&\cong\H_j(\Hom_R(\check{L}_x\otimes_R K_M^t, \E_R(R/ \fm))\\
&\cong\Hom_R(\H^j(\check{L}_x\otimes_R K_M^t), \E_R(R/ \fm))\\
&\cong\Hom_R(\H^j_{\fm}(K_M^t), \E_R(R/ \fm)).
\end{array}\]
Here, the last isomorphism follows from \cite[Theorem 7.4.4]{SC2}. Therefore, $K_M^t$ is a Cohen-Macaulay if and only if $\H_{j}(\Lam_{\fm}(\H_{\fm}^t(M)))=0$ for all $j\neq t$, as required.
\end{proof}
\end{prop}
It is well-known that a complete local ring $(R,\fm)$ with respect to $\fm$-adic completion has a dualizing complex. Therefore, we provide the following result as an immediate consequence of Corollary \ref{*1} and Proposition \ref{*}.
\begin{cor}Let $(R,\fm)$ be a complete local ring with respect to $\fm$-adic completion and $M$ be a finitely generated $R$-module. Then, the following statements are equivalent:
\begin{itemize}
\item[(i)]{{ $M$ is a sequentially Cohen-Macaulay $R$-module.}}
\item[(ii)]{{$K_M^t$ is zero or a Cohen-Macaulay $R$-module of dimension $t$ for all $0\leq t\leq \dim M$.}}
\end{itemize}
\end{cor}
In \cite[Corollary 5.4]{SC1}, for a $d$-dimensional finitely generated $R$-module $M$ over the local ring $(R,\fm)$ it is proven that $K_{\widehat M}^d$ is a Cohen-Macaulay $\widehat R$-module if and only if \emph{$\H_{j}(\Lam_{\fm}(\H_{\fm}^d(M)))=0$} for all $j\neq d$. So, the following corollary, which is a consequence of Proposition \ref{*}, is a generalization of \cite[Corollary 5.4]{SC1}.
\begin{cor}Let $(R,\fm)$ be a local ring and $M$ be a finitely generated $R$-module. Then, for any $0\leq t\leq \dim M$ the following statements are equivalent:
\begin{itemize}
\item[(i)]{{$K_{\widehat M}^t$ is zero or a Cohen-Macaulay $\widehat R$-module of dimension $t$.}}
\item[(ii)]{{ \emph{$\H_{j}(\Lam_{\fm}(\H_{\fm}^t(M)))=0$} for all $j\neq t$.}}
\end{itemize}
\begin{proof}First notice that, by \cite[Corollary 2.2.5]{SC2}, $\H_{j}(\Lam_{\fm}(\H_{\fm}^t(M)))\cong \H_{j}(\Lam_{\fm\widehat R}(\H_{\fm\widehat R}^t(\widehat M)))$ for all $j$. Hence we may and do assume that $R$ is a complete local ring, and so it has a dualizing complex. Therefore, the proof is complete by Proposition \ref{*}.
\end{proof}
\end{cor}

In \cite[Theorem 10.5.9]{SC2}, it is shown that $R$ is a $d$-dimensional Gorenstein ring if and only if $\H_{i}(\Lam_{\fm}(\E_R(k)))=0$ for all $i\neq d$ and $\H_{d}(\Lam_{\fm}(\E_R(k)))\cong \widehat{R}^{\fm}.$ So, in this direction we provide the following result:
\begin{thm}Let $(R,\fm)$ be a local ring of dimension $d$. and let $\fa$ be an ideal of $R$ . Then, the following statements are equivalent:
\begin{itemize}
\item[(i)]{{\emph{$R$} is Gorenstein}}
\item[(ii)]{{\emph{$\H_{i}(\Lam_{\fm}(\E_R(k)))=0$} for all $i\neq d$ and \emph{$\H_{d}(\Lam_{\fm}(\E_R(k)))\cong \widehat{R}^{\fm}.$}}}
\item[(iii)]{{There exist an ideal $\fa$ of $R$ with $\cd(\fa, R)=c$ such that \emph{$\H_{i}(\Lam_{\fa}(\E_R(k)))=0$} for all $i\neq c$ and \emph{$\fd_R(\H_{c}(\Lam_{\fa}(\E_R(k))))<\infty.$}}}
\end{itemize}
\begin{proof}The implication $(i)\Longrightarrow (ii)$ follows from Corollary \ref{Cor} and the fact that $R$ is Cohen-Macaulay and $\H_{\fm}^{d}(R)\cong \E_R(k)$.
The implication $(ii)\Longrightarrow (iii)$ is obvious. For the implication $(iii)\Longrightarrow (i)$, in view of \cite[Corollary 9.2.5]{SC2} one has $\H_{i}(\Lam_{\fa}(\E_R(k)))\cong\Hom_R(\H_{\fa}^i(R),\E_R(k))$ for all $i$. Therefore, $\H_{\fa}^i(R)=0$ for all $i\neq c$, and also $\H_{c}(\Lam_{\fa}(\E_R(k)))\cong\Hom_R(\H_{\fa}^c(R),\E_R(k))$. Hence, by our assumption $\fd_R(\Hom_R(\H_{\fa}^c(R),\E_R(k)))$ is finite. On the there hand, by \cite[Theorem 3.4]{MZ2}, we have $\H_{\fa}^i(\Hom_R(\H_{\fa}^c(R),\E_R(k)))=0$ for all $i\neq c$ and also by \cite[Lemma 3.5]{MZ2} one has $\H_{\fa}^c(\Hom_R(\H_{\fa}^c(R),\E_R(k)))\cong\E_R(k)$. Therefore, by \cite[Theorem 3.2]{MZ1}, $\fd_R(\E_R(k))$ is finite which it follows that $R$ is Gorenstein, as required.
\end{proof}
\end{thm}

$\mathbf{Acknowledgments}$. The author would like to thank Dr. Mohsen Gheibi and Prof. Hossein Zakeri for their valuable and profound comments during the preparation of the manuscript, and Prof. Kamaran Divaani-Aazar for having a helpful conversation with him.



\begin{thebibliography}{99}

\bibitem{ASN}
A. Atazadeh, M. Sedghi and R. Naghipour, \emph{Cohomological dimension filtration and annihilators of top local cohomology modules}, Colloquium Mathematicum, $\mathbf{139}$, (2015), 25--35.


\bibitem{CF0}
L.W. Christensen, \emph{Gorenstein Dimensions,} Lecture notes in Mathematices, Springer-Verlag, Berlin, 2000.

\bibitem{CF1}
L.W. Christensen and H-B. Foxby, \emph{Hyperhomological Algebra with Applications to Commutative Rings,} in preparation.

\bibitem{CN}
N.T. Cuong and T.T. Nam, \emph{The I-adic completion and local homology for Artinian modules,}
Math. Proc. Cambridge Philos. Soc., $\mathbf(131)$(1), (2001), 61--72.

\bibitem{GMMM}
J. P. C. Greenlees and J. P. May, \emph{Derived functors of I-adic completion and local homology}, Journal of Algebra $\mathbf(149)$, (1992), 438--453.

\bibitem{HSC1}
M. Hellus and P. Schenzel, \emph{On cohomologically complete intersections,} J. Algebra, $\mathbf{320}$, (2008), 3733--3748.

\bibitem{HSC}
M. Hellus and P. Schenzel, \emph{Notes on local cohomology and duality}, J. Algebra $\mathbf{401}$, (2014), 48--61.

\bibitem{HEP}
J. Herzog, D. Popescu, \emph{Finite filtrations of modules and shellable multicomplexes,} Manuscripta Math. $\textbf{121}$, (2006), 385--410.


\bibitem{KLS}
B. Kubik, M. J. Leamer, and S. Sather-Wagstaff, \emph{Homology of artinian and mini-max modules,
I}, J. Pure Appl. Algebra $\textbf{215}$, no. 10, (2011), 2486--2503.

\bibitem{LIP}
J. Lipman, \emph{Lectures on local cohomology and duality, Local cohomology and its applications} (Guanajuato, 1999),
Lecture Notes in Pure and Appl. Math., 226, Dekker, New York, (2002), 39--89.
\bibitem{EM}
E. Matlis, \emph{The Koszul complex and duality}, Communications in algebra, $\mathbf(1)$(2), (1974), 87--144.


\bibitem{ARH1}
P. L. Majd and A. Rahimi, \emph{On the structure of sequentially Cohen-Macaulay bigraded modules}. Czech Math J, $\mathbf{65}$, (2015), 1011--1022.


\bibitem{ARH}
A. Rahimi, \emph{Sequentially Cohen-Macaulayness of bigraded modules,} Rocky Mountain J. Math., $\mathbf{65}$(2), 2017, 621--635.

\bibitem{MZ}
M. Rahro Zargar and H. Zakeri, \emph{On injective and Gorenstein injective dimensions of local cohomology modules}, Algebra Colloq. $\textbf{22}$, (2015), 935--946.

\bibitem{MZ2} M. Rahro Zargar, \emph{Relative canonical module and some duality results}, Algebra Colloq. $\mathbf{26}$(02), (2019), 351--360.


\bibitem{MZ1} M. Rahro Zargar and H. Zakeri, \emph{On flat and Gorenstein flat dimensions of local cohomology modules}, Canad. Math. Bull. $\textbf{59}$, (2016), 403--416.

\bibitem{ROB}
R.N. Roberts,\emph{ Krull-dimension for Artinian modules over quasi local commutative Rings}, Quart. J. Math. Oxford Ser. $\textbf{26}$(2), no. 103, (1975), 269--273.

\bibitem{JR}
J.J. Rotman, \emph{An introduction to homological algebra}, Second ed., Springer, New York, 2009.
\bibitem{SC3}
P. Schenzel, \emph{On the dimension filtration and Cohen-Macaulay filtered modules,} in: F. Van Oystaeyen (Ed.),
Commutative Algebra and Algebraic Geometry, in: Lecture Notes Pure Appl. Math., vol. 206, Dekker, New
York, 1999, pp. 245--264.

\bibitem{SC1}
P. Schenzel, \emph{Notes on endomorphisms, local cohomology and completion}, arXiv preprint arXiv:2105.00664, 2021 - arxiv.org.

\bibitem{SC2}
P. Schenzel and A. M. Simon, \emph{Completion, ¡$\check{C}$ech and local homology and cohomology,} Interactions between them.,
Cham: Springer, 2018.

\bibitem{STA}
R.P. Stanley, \emph{Combinatorics and Commutative Algebra}, Birkh\"{a}user, Basel, 1983.

\bibitem{ZT}
Z. Tang, \emph{Local Homology and Local Cohomology}, Algebra Colloq. $\textbf{11}$(4), (2004),  467--476.
\bibitem{TZA}
Z. Tang and H. Zakeri, \emph{Co-Cohen-Macaulay modules and
modules of generalized fractions,} Communications in Algebra, $\textbf{22}$, no. 6, (1994), 2173--2204.
















\end{thebibliography}
\end{document}